\documentclass[reqno, leqno,11pt]{article}
\usepackage{amsmath,amssymb,amsthm,amsbsy}
\usepackage[english]{babel}
\usepackage[ansinew]{inputenc}
\usepackage{graphicx}
\usepackage{a4wide}  
\usepackage{xcolor}
\usepackage{bbm,  mathrsfs}
\usepackage{scalerel,stackengine}
\usepackage{mathtools}
\usepackage[normalem]{ulem}
\usepackage{fancyhdr}

\numberwithin{equation}{section}

\newtheorem{theorem}{Theorem}
\newtheorem{lemma}[theorem]{Lemma}
\newtheorem{definition}[theorem]{Definition}

\theoremstyle{remark}
\newtheorem{rem}{Remark}[section]

\newcommand{\e}{\mathrm{e}} 
\newcommand{\N}{\mathbb{N}}

\newcommand{\Z}{\mathbb{Z}}

\stackMath
\newcommand\reallywidehat[1]{%
\savestack{\tmpbox}{\stretchto{%
  \scaleto{%
    \scalerel*[\widthof{\ensuremath{#1}}]{\kern.1pt\mathchar"0362\kern.1pt}%
    {\rule{0ex}{\textheight}}
  }{\textheight}%
}{2.4ex}}%
\stackon[-6.9pt]{#1}{\tmpbox}%
}
\def\multiset#1#2{\ensuremath{\left(\kern-.6em\left(\genfrac{}{}{0pt}{}{#1}{#2}\right)\kern-.6em\right)}}
\newcommand{\mch}[2]{
\left.\mathchoice
  {\left(\kern-0.48em\binom{#1}{#2}\kern-0.48em\right)}
  {\big(\kern-0.30em\binom{\smash{#1}}{\smash{#2}}\kern-0.30em\big)}
  {\left(\kern-0.30em\binom{\smash{#1}}{\smash{#2}}\kern-0.30em\right)}
  {\left(\kern-0.30em\binom{\smash{#1}}{\smash{#2}}\kern-0.30em\right)}
\right.}

\newcommand{\be}{\begin{equation}}
\newcommand{\ee}{\end{equation}}

\def\1{{\mathchoice {1\mskip-4mu\mathrm l}      
{1\mskip-4mu\mathrm l}
{1\mskip-4.5mu\mathrm l} {1\mskip-5mu\mathrm l}}}

\title{Logarithms of Catalan generating functions: A combinatorial approach}
\author{Sabine Jansen\thanks{Mathematisches Institut, Ludwig-Maximilians-Universit{\"a}t, Theresienstr. 39, 80333 M{\"u}nchen, Germany.}\ \thanks{Munich Center for Quantum Science and Technology (MCQST), Schellingstr. 4, 80799 M{\"u}nchen.} \and Leonid Kolesnikov\footnotemark[1]}
\date{}
\begin{document}




\maketitle
\thispagestyle{empty}
\begin{abstract}	
We analyze the combinatorics behind the operation of taking the logarithm of the generating function $G_k$ for $k^\text{th}$ generalized Catalan numbers. We provide combinatorial interpretations in terms of lattice paths and in terms of tree graphs. Using explicit bijections, we are able to recover known closed expressions for the coefficients of $\log G_k$ by purely combinatorial means of enumeration. The non-algebraic proof easily generalizes to higher powers $\log^a G_k$, $a\geq 2$.~\\

	

\end{abstract}

\section{Introduction}
The present article originated in the following question: given $k\in \N$, what is the combinatorial interpretation of the power series $F(x)$ that solves the equation
\begin{equation}
\label{eq:question}
	\e^{F(x)} = 1+ x\, \e^{k F(x)},
\end{equation} 
and is there a way of computing the coefficients of $F(x)$ by counting suitable labeled combinatorial structures? The question was raised in the context of statistical mechanics for a one-dimensional system of non-overlapping rods on a line~\cite[Section 5.2]{jansen2015tonks}; up to sign flips, the function $F(x)$ corresponds to the pressure of a gas of rods of length $k$ and activity $x$ on the discrete lattice $\Z$. 

The exponential $\exp(F(x))$ is easily recognized as the generating function for (generalized) Catalan numbers, whose definition we recall below. Thus we are looking for a combinatorial interpretation of the logarithm of the generating function for (generalized) Catalan numbers.
Logarithms of Catalan generating functions have in fact attracted interest since Knuth's Christmas lecture~\cite{knuth_christmas2014}; to the best of our knowledge, the focus has been on the computation of coefficients, with the question of combinatorial interpretation left open. 
 We provide several such interpretations, among them one with cycle-rooted labeled trees. For the interpretation it is essential that we work with \emph{labeled} combinatorial species, as is manifest already for a simple special case: For $k=1$, the solution to~\eqref{eq:question} is 
\[
	F(x) = - \log (1-x) = \sum_{n=1}^\infty \frac{x^n}{n} = \sum_{n=1}^\infty \frac{x^n}{n!}\, (n-1)!.
\] 
As $1/n$ is not an integer, the function $F$ is not an ordinary generating function, but it is the exponential generating function for a labeled structure, namely for cycles.

Let us recall some facts about Catalan numbers.
The sequence of natural numbers ($C_n)_{n\geq 0}$ with
$$C_n:=\frac{(2n)!}{(n+1)!n!},\quad n\geq 0,$$
is commonly referred to as Catalan numbers since the 1970's. The name goes back to Eug{\`e}ne Charles Catalan who was the first to introduce Catalan numbers in the above form, after they already appeared in literature as far back as the 18th century, most prominently in the work of Leonhard Euler.

Catalan numbers emerge in a huge variety of different counting problems: Over 200 possible interpretations are listed in the monograph~\cite{stanley_2015} by R. P. Stanley alone; many of those are of great significance in the field of combinatorics. Two especially prominent types of structures enumerated by Catalan numbers are discrete paths (e.g., Dyck or Motzkin paths) and tree graphs (e.g., binary or plane trees) under certain restrictions,  see items $4-56$ in~\cite[Chapter~2]{stanley_2015}).

The generating function $G_2$ of Catalan numbers $(C_n)_{n\geq 0}$ is given by the formal power series 
$$
G_2(x):=\sum_{n\geq 0}C_n x^n=1+\sum_{n\geq 1}\frac{x^n}{n!}\frac{(2n)!}{(n+1)!}. 
$$
Naturally, one can view $G_2$ as the ordinary generating function for any of the over 200 unlabeled structures in~\cite{stanley_2015} or as the exponential generating function for any of the corresponding labeled structures (in the sense of combinatorial species and associated generating functions, see~\cite{bergeron-labelle-leroux1998book}). In particular, we will view $G_2$ as the exponential generating function for labeled lattice paths (see Section 2) or for labeled binary trees (see Section 3). 

The generating function $G_2$ can be generalized to the following formal power series: For $k\geq 2$, consider the power series $G_k$, sometimes called the \emph{binomial series} \cite{chu2019,prod2019}, given by
$$
G_k(x):=1+\sum_{n\geq 1}\frac {x^n}{n}\binom{kn}{n-1}=1+\sum_{n\geq 1}\frac {x^n}{n!}(n-1)!\binom{kn}{n-1} 
$$ 
and let us refer to the coefficients $$\frac{1}{n}\binom{kn}{n-1},\quad n\geq 1,$$ as \emph{generalized $k^\text{th}$ Catalan numbers} following the terminology in~\cite{hiltped91} (also known under the name of Fuss-Catalan numbers~\cite{lin2011}); notice that the Catalan numbers $(C_n)_{n\geq 1}$ are indeed recovered for $k=2$. The power series $G_k$ satisfies
\[
	G_k(x) = 1+ x G_k(x)^k.
\]
It is well-known (see~\cite{hiltped91}) that generalized $k^\text{th}$ Catalan numbers enumerate monotone lattice paths, the so called $k$-good paths, or alternatively plane $k$-ary trees. Therefore, we can and will interpret $G_k$ as the exponential generating function for labeled lattice paths (see Section 2) or for labeled plane $k$-ary trees (see Section 3).

The main object of study in this paper is the logarithm of the generating function $G_k$ for $k\geq 2$, which again can be represented by a formal power series. Explicit expressions for the coefficients are already known from the literature:  The expansion of $\log G_2$ was presented 2014 in the annual Christmas lecture by Donald Knuth~\cite{knuth_christmas2014} --- who subsequently posed an elegant conjecture for the expansion of $\log^2 G_2$ as a problem in~\cite{knuth2015prob} to be solved by various authors soon after:
\begin{align*}\label{formula:knuth_squared}
\log^2 G_2(x)=\sum_{n\geq 2}\frac{x^n}{n}\binom{2n}{n}(H_{2n-1}-H_n),
\end{align*}
where the \emph{harmonic numbers} $(H_m)_{m\in\N}$ are given by $H_m:=\sum_{i=1}^m \frac{1}{i}$ for $m\in\N$. 

Higher powers $\log^a G_k$, $a\geq 2$, were examined in~\cite{chu2019} and~\cite{prod2019}, explicit formulas for the coefficients were derived --- in terms of harmonic numbers in the former and in terms of Stirling cycle numbers in the latter work. The proofs are of algebraic nature and involve general inversion formulas --- in particular, the Lagrange inversion formula. 

%
 
Here, we present a combinatorial, bijective proof providing explicit expressions for the coefficients of $\log G_k$ by means of exact enumeration. The proof easily generalizes to the case of the higher powers $\log^a G_k$, $a\geq 2$. For example, in the aforementioned case of the squared logarithm $\log^2 G_k$, we obtain $$[x^n]\log^2G_k(x)=2\sum_{p=2}^n \frac{k-1}{kn-p}\binom{kn-p}{n-p}H_{p-1}.$$

Naturally, this expression for the coefficients can be rewritten to match the one by Knuth presented above. In the general case $a\geq 1$, we get the formula
$$
[x^n]\log^aG_k(x)=\sum_{p=a}^n c^{(p)}_{k,n}N_{p,a},
$$
where 
$$
c^{(p)}_{k,n}=\frac{kp-p}{kn-p}\binom{kn-p}{n-p}
$$
and, using $[n]:=\{1,\ldots,n\}$ for $n\in\N$,
$$
N_{p,a}:=\sum_{\substack{(q_1,\ldots,q_a)\in [p]^a\\q_1+\ldots+q_a=p}}\frac{1}{\prod_{i=1}^{a} q_i}.
$$

While identifying the coefficients of $\log^a G_k$ for $k\geq 2$ and $a\geq 1$ is not a novel result (since those are known from~\cite{chu2019} and~\cite{prod2019}), we think that our proof itself is of interest --- as we are not aware of any alternative proof that is essentially non-algebraic in nature.

The article is organized as follows: In Section 2, we provide a combinatorial interpretation of $\log G_k$ in terms of lattice paths (Theorem~\ref{thm:bij_paths}) by using bijective results identifying lattice paths with sets of certain paths or path-like structures (Lemma~\ref{lem:bb} and Lemma~\ref{lem:bij:orn_to_min}). Additionally, we use this interpretation to provide a closed expression for coefficients of $\log G_k$ (Theorem~\ref{thm:orn_count}) via a purely combinatorial proof, which can be easily generalized to higher powers $\log^a G_k$, $a\geq 2$ (Theorem~\ref{thm:orn_count_h}). In Section 3 we provide an alternative interpretation of $\log G_k$ in terms of plane trees (Theorem~\ref{thm:main1}). Again, at the heart of this interpretation is a bijective result identifying $k$-ary trees with sets of certain trees or tree-like structures (Lemma~\ref{lem:bij:tree_to_forest} and Lemma~\ref{lem:bij:circ_to_min}). Finally, in the appendix, a method to encode both lattice paths and plane trees via \emph{cyclically ordered multisets} is introduced, providing a bijection between the two combinatorial species and establishing a direct connection between the two combinatorial interpretations of $\log G_k$.

\section{Combinatorial interpretation via lattice paths}\label{sec:2}

\subsection{Lattice paths and associated generating functions}
In this section, we want to consider a combinatorial interpretation of (generalized) Catalan numbers in terms of monotone lattice paths and understand the logarithm of the corresponding generating functions on the level of these combinatorial structures. We concentrate on item $24$ in~\cite[Chapter~2]{stanley_2015}, but consider labeled structures instead of unlabeled.

\begin{definition}[Labeled good paths]
Let $n\in\N$ and let $k\geq 2$. Let $V\subset \N$ be a finite label set with $\vert V\vert=n$. A path on the quadratic lattice $\Z^2$ from $(0, 0)$ to $(n, (k-1)n)$ with steps $(0, 1)$ or $(1, 0)$, together with a labeling of the heights $\{(k-1)j\}_{0\leq j\leq n-1}$ by elements of $V$ (as visualized in Figure~\ref{fig:path_repr_p=2}), is called a $V$-labeled $k$-good path if it never rises above the line $y = (k-1)x$. Denote  the set of all such paths by $\mathscr P_k(V)$ and write $\mathscr P_k(n):=\mathscr P_k([n])$.
\end{definition}

\begin{rem}
By labeling we mean a bijective map from $\{(k-1)j\}_{0\leq j\leq n-1}$ to $V$. Notice these heights are exactly those on which the path can potentially intersect the diagonal $y=(k-1)x$.
\end{rem}

\begin{rem} 
Our notion of (unlabeled) good paths is essentially the same as introduced in \cite{hiltped91}, up to a vertical shift of the path by 1. Notice that, by~\cite{hiltped91}, $G_k$ --- as the generating function for $k^\text{th}$ generalized Catalan numbers ---  is equal to the exponential generating function for $(\mathscr P_k(n))_{n\in\N_0}$, i.e.,
$$
	G_k(x)=1+\sum_{n\geq 1}\frac{x^n}{n!}\vert \mathscr P_k(n) \vert.
$$ 
\end{rem}
\begin{figure}[ht]
\centering
\includegraphics[scale=0.4]{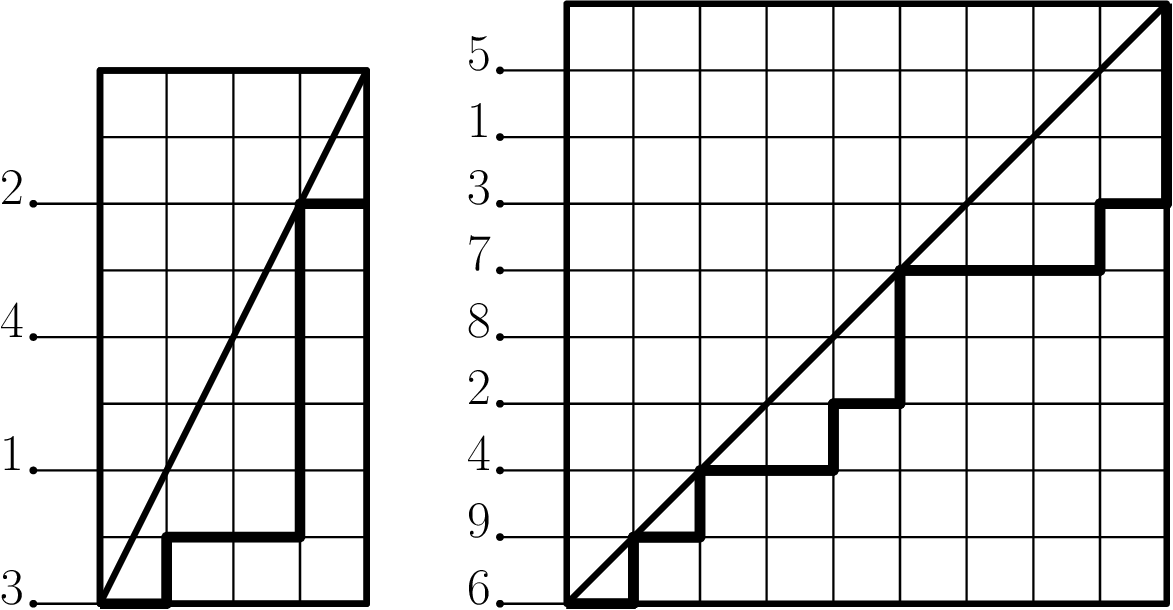}
\caption{On the left side, we see a $3$-good lattice path of length $12$ (labeled by $\{1,2,3,4\}$); on the right side, we see a $2$-good lattice path of length $18$ (labeled by $\{1,\ldots, 9\})$.}\label{fig:path_repr_p=2}
\end{figure}

Next, we want to introduce combinatorial structures that are enumerated by the coefficients of $\log G_k$.

\begin{definition}[Label-minimal good paths] Let $n\in\N$ and let $k\geq 2$. Let $V\subset \N$ be a finite label set with $\vert V\vert=n$. A $V$-labeled $k$-good path $P$ is called label-minimal if the label of the height 0 is minimal under all labels labeling heights at which $P$ intersects the diagonal $y = (k-1)x$. 

Denote the set of $V$-labeled $k$-good paths that are label-minimal by $\mathscr P^{\min}_k(V)$ and write $\mathscr P^{\min}_k(n):=\mathscr P^{\min}_k([n])$. The corresponding exponential generating function is defined by the following formal power series:
$$
G^{\min}_k(x)=\sum_{n\geq 1}\frac{x^n}{n!}\vert\mathscr P^{\min}_k(n)\vert.
$$

Let $1\leq \ell\leq n$ and let $B_1\cup\ldots\cup B_\ell$ be a partition of $[n]$. For every $i\in[\ell]$, let $P_i\in \mathscr P^{\min}_k(B_i)$. The set $\{P_1,\ldots, P_\ell\}$ is called a label-minimal $k$-field on $[n]$. Denote the set of all label-minimal $k$-fields on $[n]$ by $\mathscr F^{\min}_k(n)$.
\end{definition}

\begin{figure}[ht]
\centering
\includegraphics[scale=0.4]{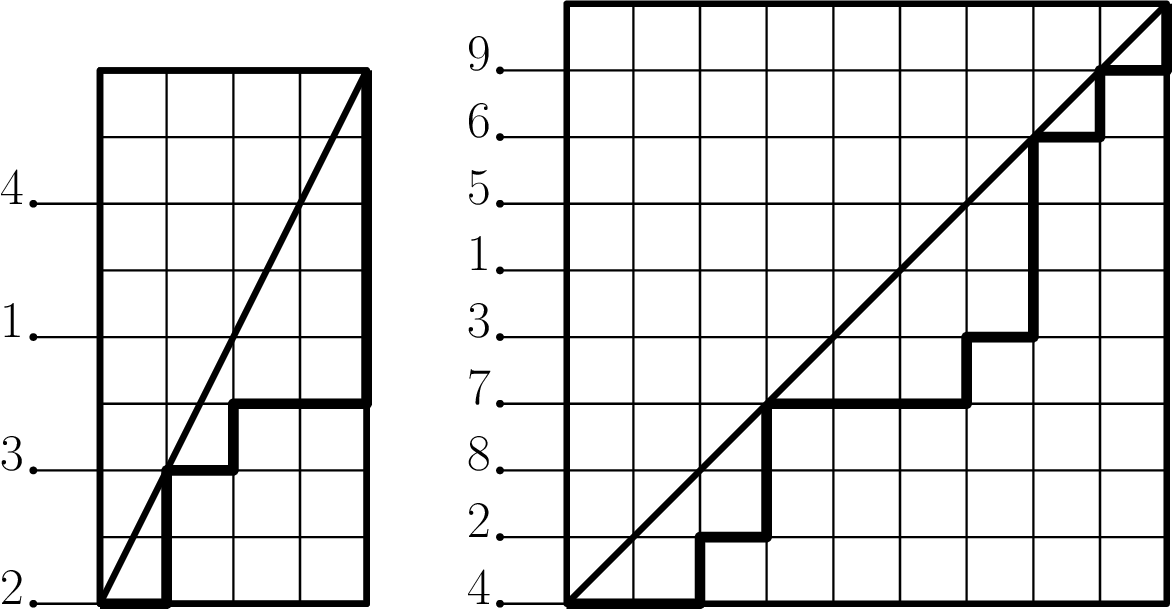}
\caption{On the left side, we see a label-minimal $3$-good lattice path of length $12$ (labeled by $\{1,2,3,4\}$); on the right side, we see a label-minimal $2$-good lattice path of length $18$ (labeled by $\{1,\ldots, 9\})$.}\label{fig:mpath_repr_p=2}
\end{figure}

Alternatively, just like the logarithm of the exponential generating function for permutations can be interpreted as the exponential generating function for cycles (as explained in the introduction), one can interpret $\log G_k$ via certain cyclic structures as well. Informally speaking, those cyclic structures can be obtained by ``bending $k$-good paths into circles", i.e., by identifying endpoints of $[n]$-labeled $k$-good paths with their starting points and keeping the labelings (which thus become cycles on $[n]$).
 
\begin{definition}[Labeled ornaments]\label{def:ornaments}
Let $n\in\N$ and let $k\geq 2$. Let $V$ be a finite label set with $\vert V\vert=n$. For $P\in\mathscr P_k(V)$ construct a labeled infinite lattice path $\hat P$ by taking (infinitely many) labeled paths $j(n,(k-1)n)+P$, $j\in\Z$, and concatenating them (while keeping the labeling). 

An equivalence relation on the set $\mathscr P_k(V)$ can be defined as follows: Let two $V$-labeled $k$-good paths $P_1$ and $P_2$ be equivalent if and only if $\hat P_1$ is a translate of $\hat P_2$ along the line $y=(k-1)x$ (including the labeling).

The corresponding equivalence classes $[P]$ can be identified with the shape of the infinite periodic paths $\hat P$ together with an infinite periodic labeling (i.e., a cycle on $[n]$) which are obtained by identifying the endpoint and the starting point of $P$.

Denote the set of the equivalence classes, called $V$-labeled $k$-ornaments, by $\mathscr P^{\circ}_k(V)$ and write $\mathscr P^{\circ}_k(n):=\mathscr P^{\circ}_k([n])$. The corresponding exponential generating function is defined by the following formal power series:
$$
G^\circ_k(x)=\sum_{n\geq 1}\frac{x^n}{n!}\vert\mathscr P^{\circ}_k(n)\vert.
$$

Let $1\leq \ell\leq n$ and let $B_1\cup\ldots\cup B_\ell$ be a partition of $[n]$. For every $i\in[\ell]$, let $O_i$ be a $B_i$-labeled $k$-ornament. The set $\{O_1,\ldots, O_\ell\}$ is called a $k$-ornament field on $[n]$. Denote the set of all $k$-ornament fields on $[n]$ by $\mathscr F^\circ_k(n)$.
\end{definition}

\begin{figure}[ht]
\centering
\includegraphics[scale=0.4]{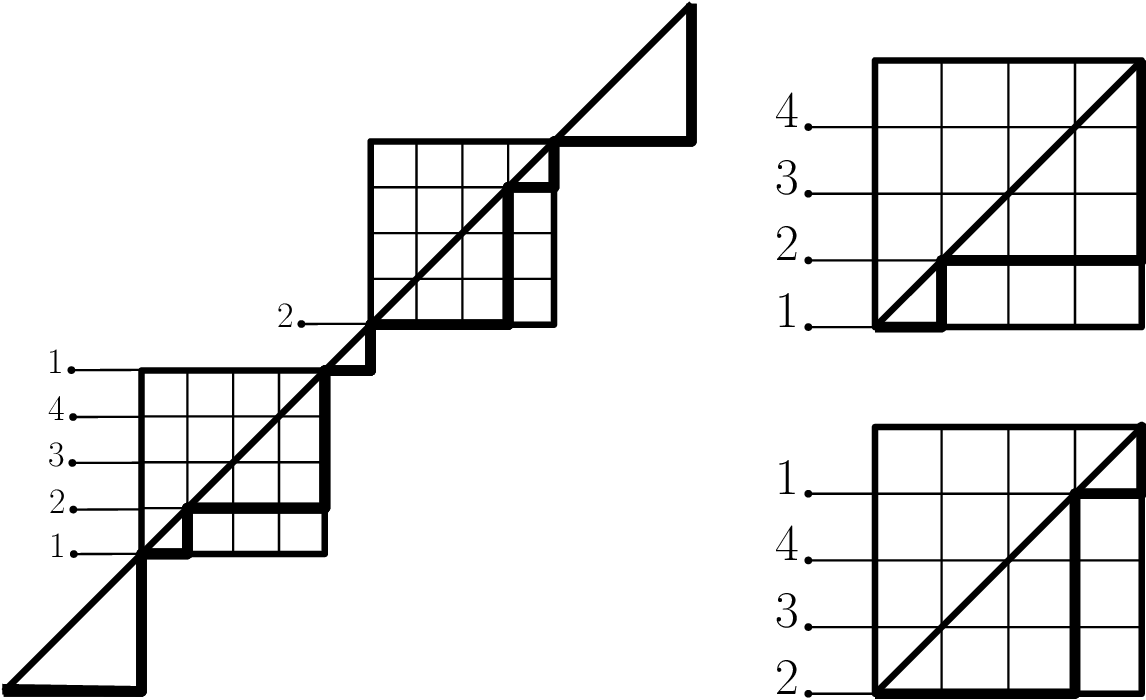}
\caption{Both $2$-good paths of length $8$ depicted on the right side are representatives of the 2-ornament depicted on the left side.}\label{fig:orn_repr}
\end{figure}

Now, with these definitions at hand, we are ready to give a combinatorial interpretation for $\log G_k$ in terms of label-minimal $k$-good paths or, alternatively, in terms of $k$-ornaments. 

\subsection{Bijective results}

The following lemma provides the combinatorial insight essential to the proofs of the main results in this section: It enables us to identify labeled good paths with sets of label-minimal good paths.

\begin{lemma}\label{lem:bb}Let $n\in\N$ and let $k\geq 2$. There is a bijection between $\mathscr P_k(n)$ and $\mathscr F^{\min}_k(n)$.  
\end{lemma}
\begin{proof}
Let us define a bijection $m:\mathscr P_k(n)\to\mathscr F^{\min}_k(n)$. For a $k$-good path $P\in\mathscr P_k(n)$, we obtain a label-minimal $k$-field $m(P)\in \mathscr F^{\min}_k(n)$ from $P$ by the following inductive procedure:

\begin{enumerate}
\item [Step $0$:] Set $\Pi=P$.  
\item [Step $N\geq 1$:] 

Let $0=y_1<\ldots<y_\ell$ denote the heights at which the path $\Pi$ intersects the line $y=(k-1)x$ and let $i_1,\ldots,i_\ell\in[n]$ denote the corresponding labels. 
\begin{itemize}
\item If there exists a $j\in[\ell]$ such that $i_j<i_1$ holds, set $y:=\min \{y_j:j\in[\ell],\ i_j<i_1\}$ and set $\Pi=P$. Cut the path $\Pi$ at the height $y$, obtaining two paths --- a path $\Pi_1$ from $(0,0)$ to $(\frac{y}{k-1},y)$ and a path $\Pi_2$ starting at $(\frac{y}{k-1},y)$  which inherit their labelings from $\Pi$. $\Pi_1$ and $\Pi_2$ are again $k$-good paths --- up to a translation of $\Pi_2$. Replace $\Pi$ with the translate of $\Pi_2$ starting in $(0,0)$ and GOTO Step N+1.

\item Otherwise STOP.
\end{itemize}
\end{enumerate}

 Naturally, this procedure produces a label-minimal $k$-field on $[n]$. 

Conversely, given a label-minimal $k$-field $F\in\mathscr F^\circ_k(n)$, construct an $[n]$-labeled $k$-good path $m^{-1}(F)\in\mathscr P_k(n)$ as follows: Order the labeled $k$-good paths from $F$ decreasing in the label at $y=0$. Successively, glue the predecessor path to the successor path by concatenation (identifying the endpoint of the former with the starting point of the latter). Naturally, the resulting lattice path is an $[n]$-labeled $k$-good path and the described procedure does indeed define the inverse of the map $m$ introduced above.
\end{proof}

\begin{rem}\label{rem:sa_1}
Clearly, our choice of label-minimal paths is somewhat arbitrary in the following sense: In the inductive procedure from Lemma~\ref{lem:bb} defining the map $m$, one can choose different rules to ``cut'' the path $P$ at its intersections with the diagonal. E.g., one could instead consider ``label-maximal'' paths (or, more generally, define $y$ as the height labeled minimally with respect to an arbitrary order on the labels instead of the canonical one).
\end{rem}

\begin{figure}[ht]
\centering
\includegraphics[scale=0.4]{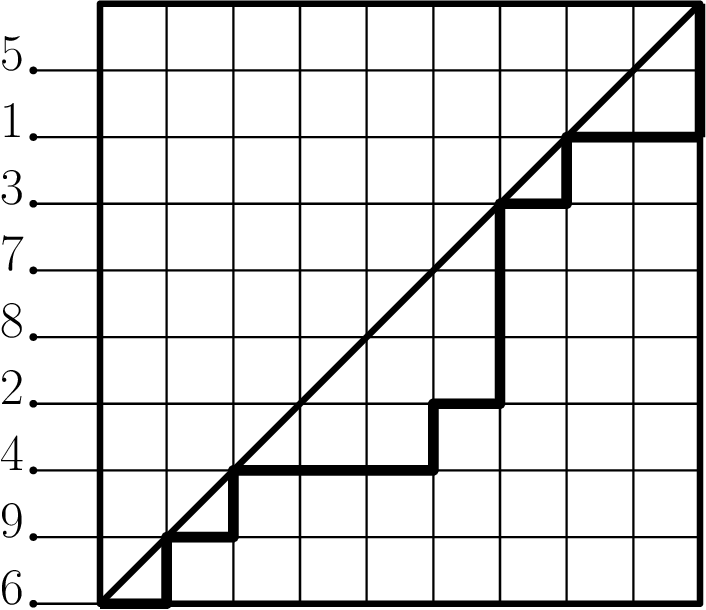}\qquad\includegraphics[scale=0.4]{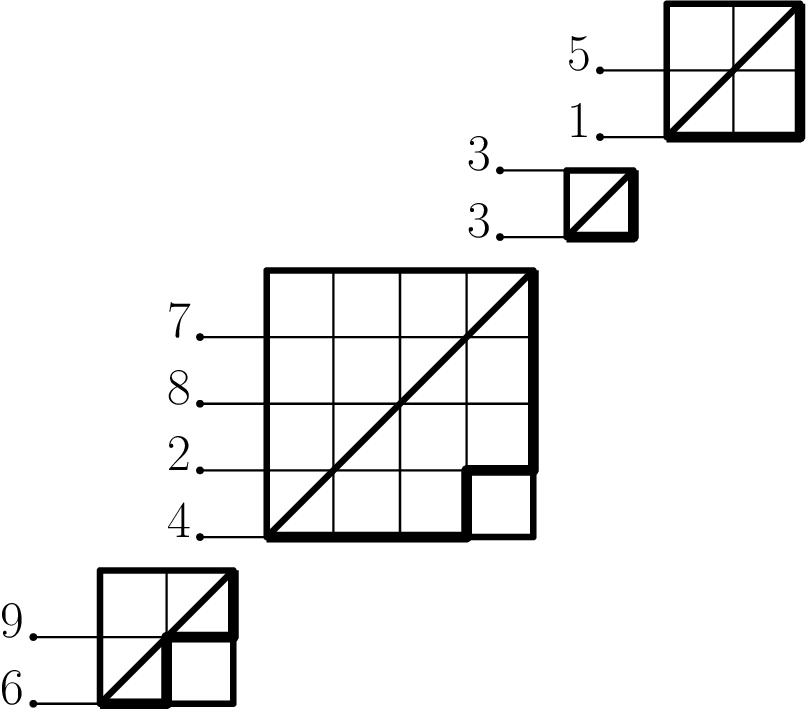}
\caption{On the left, we see a $2$-good path of length 18, on the right we see the label-minimal 2-field corresponding to it in the sense of the proof of Lemma~\ref{lem:bb}.}\label{fig:orn_bij}
\end{figure}

This bijective result allows us to interpret $\log G_k$ as the exponential generating function for label-minimal good paths: 

\begin{theorem}\label{thm:bij_paths}Let $k\geq 2$. The following holds as an identity between formal power series:
$$\log G_k=G^{\min}_k.$$ 
\end{theorem}
\begin{proof}
The claim follows directly from Lemma~\ref{lem:bb} via a standard combinatorial argument (see, e.g.,~\cite{faris2010combinatorics}, for the argument formulated in the framework of combinatorial species). 
\end{proof}

\begin{lemma}\label{lem:bij:orn_to_min}
For $n\in\N$ and $k\geq 2$, there is a bijection between the sets $\mathscr P^{\min}_k(n)$ and $\mathscr P^{\circ}_k(n)$.
\end{lemma}
\begin{proof}
The bijection is given by assigning to the label-minimal path $P\in \mathscr P^{\min}_k(n)$ its equivalence class $[P]\in \mathscr P^{\circ}_k(n)$. This map is clearly invertible since every element of $\mathscr P^{\circ}_k(n)$ has a unique representative $P\in \mathscr P^{\min}_k(n)$ that is label-minimal. 
\end{proof}
\begin{rem}
Again, we see that the choice of label-minimal paths was somewhat arbitrary: In the above proof, one could identify $[P]\in\mathscr P^{\circ}_k(n)$ with a representative different from $P$, e.g., with the ``label-maximal'' path in $[P]$, see Remark~\ref{rem:sa_1}.
\end{rem}

The lemma allows us to identify $\log G_k$ with the exponential generating function for labeled $k$-ornaments:
\begin{theorem}\label{thm:bij_paths_2}Let $k\geq 2$. The following holds as an identity between formal power series:
$$\log G_k=G^\circ_k.$$ 
\end{theorem}
\begin{proof}
The claim follows from Theorem~\ref{thm:bij_paths} and Lemma~\ref{lem:bij:orn_to_min} since the latter implies that $G^{\circ}_k=G_k^{min}$ for $k\geq 2$.
\end{proof}

We have shown how taking the logarithm of the generating function for $k^{\text{th}}$ Catalan numbers $G_k$ can be interpreted on the level of lattice paths. By Theorem~\ref{thm:bij_paths}, it can be interpreted as the exponential generating function for label-minimal $k$-good paths --- so that taking the logarithm of $G_k$ corresponds to discarding those $k$-good paths that have labels at height $0$  which are not minimal among the labels labeling intersections of the path with the diagonal $y=(k-1)x$. Alternatively, by Theorem~\ref{thm:bij_paths_2}, $\log G_k$ can be interpreted as the exponential generating function for $k$-ornaments --- so that taking the logarithm corresponds to identifying those $k$-good paths that result in the same $k$-ornament when they are ``bent into a circle''.

\subsection{Identifying the coefficients}
Lemma~$\ref{lem:bb}$ also provides an elementary way to recover the explicit expressions for the coefficients of $\log G_k$ for every $k\geq 2$ (known from~\cite{jansen2015tonks, knuth_christmas2014}) --- by simply counting $k$-ornaments.

\begin{theorem}\label{thm:orn_count}
Let $k\geq 2$. We have $$\log G_k(x)=\sum_{n\geq 1} \frac{x^n}{n!}\frac{(kn-1)!}{(kn-n)!}.$$
\end{theorem}
\begin{proof}
A well-known result (see, e.g.,~\cite{drube2022}) provides the number $c^{(p)}_{k,n}$ of Dyck paths of length $kn$ with exactly $p\in\N$ returns to zero (which corresponds to the number of unlabeled $k$-good paths of length $kn$ that intersect the diagonal $y=(k-1)x$ exactly $p+1$ times):
$$c^{(p)}_{k,n}=\frac{kp-p}{kn-p}\binom{kn-p}{n-p}.$$
Notice that if some $k$-good lattice path $P$ intersects the line $y=(k-1)x$ exactly $p+1$ times then the same holds for every path in $[P]$ and $\vert[P]\vert=p$ (since choosing a representative of $[P]$ is equivalent to choosing which intersection point to place at $y=0$). Therefore, the number of $[n]$-labeled $k$-ornaments intersecting the diagonal $y=(k-1)x$ exactly $p$ times (for any representative, counting starting point and endpoint as one intersection) is given by $$n!\frac{c^{(p)}_{k,n}}{p}$$ and  thus
we get 
$$
[x^n]\log G_k(x)=\sum_{p=1}^n \frac{c^{(p)}_{k,n}}{p}=\sum_{p=1}^n \frac{k-1}{kn-p}\binom{kn-p}{n-p}=\frac{1}{n!}\frac{(kn-1)!}{(kn-n)!}.
$$
\end{proof}

The presented proof of the preceding theorem has the following advantage: It can be easily modified to investigate the coefficients of $\log^a G_k$ for higher powers $a\geq 2$. As mentioned in the introduction, the result itself is not novel and similar expressions for the coefficients are known from~\cite{chu2019, prod2019}. 

\begin{theorem}\label{thm:orn_count_h}
Let $a, k\geq 2$ and $n\in\N$. We have
$$
[x^n]\log^aG_k(x)=\sum_{p=a}^n c^{(p)}_{k,n}N_{p,a},
$$
where 
$$
c^{(p)}_{k,n}=\frac{kp-p}{kn-p}\binom{kn-p}{n-p}
$$
and
$$
N_{p,a}:=\sum_{\substack{(q_1,\ldots,q_a)\in [p]^a\\q_1+\ldots+q_a=p}}\frac{1}{\prod_{i=1}^{a} q_i}.
$$
\end{theorem}
\begin{rem}In the special case $a=2$, considered by Knuth in~\cite{knuth2015prob}, we get 
$$[x^n]\log^2G_k(x)=\sum_{p=2}^n c^{(p)}_{k,n}\sum_{1\leq q\leq p-1} \frac{1}{q(p-q)}=2\sum_{p=2}^n \frac{k-1}{kn-p}\binom{kn-p}{n-p}H_{p-1},$$
where $(H_m)_{m\in\N}$ are the harmonic numbers defined in the introduction.
\end{rem}
\begin{proof} By Theorem~\ref{thm:bij_paths} and by a standard combinatorial argument, $\log^a G_k$ is the exponential generating function for $k$-ornament fields consisting of $a\geq 2$ $k$-ornaments. For every $n\in\N$, we need to determine the number of such $k$-ornament fields on $[n]$. To do so, we employ the same decomposition as in the proof of Theorem~\ref{thm:orn_count} sorting the k-ornament fields by the total number of intersections with the diagonal $y=(k-1)x$ (in any corresponding set of representatives). So, let  $\hat N^{(k,n)}_{p,a}$ denote the number of $k$-ornament fields on $[n]$ consisting of precisely $a\geq 2$ $k$-ornaments such that in total there are $p$ intersections with the diagonal $y=(k-1)x$ (for any representative, counting starting point and endpoint as one intersection). Then
$$
[x^n]\log^a G_k(x)=\frac{a!}{n!}\sum_{p=a}^n\hat N^{(k,n)}_{p,a}.
$$
In the proof of Theorem~\ref{thm:orn_count}, we already established that $$n!\frac{c^{(p)}_{k,n}}{p}$$ is the number of $[n]$-labeled $k$-ornaments $O$ intersecting the diagonal $y=(k-1)x$ exactly $p$ times (for any representative, counting starting point and endpoint as one intersection). We now want to determine how many $k$-ornament fields of precisely $a\geq 2$ $k$-ornaments correspond to each such $k$-ornament $O$ --- in the sense that they can be obtain by cutting $O$ at precisely $a\geq 2$ intersections with the diagonal $y=(k-1)x$. This number is exactly the number of possible decompositions of a cycle of length $p$ into $a\geq 2$ segments which is given by
$$
\sum_{\substack{(q_1,\ldots,q_a)\in [p]^a\\q_1+\ldots+q_a=p}}\frac{p}{a},
$$
where the tuple $(q_1,\ldots,q_a)$ corresponds to the lengths of the segments, the factor $p$ corresponds to the possible choice of the starting point for the first segment and the factor $\frac{1}{a}$ is due to the fact that there are $a\geq 2$ sequences $(q_1,\ldots,q_a)$ corresponding to the same cycle on $\{q_1,\ldots,q_a\}$.

Left to notice is the following: Consider a $k$-ornament field of $a\geq 2$ $k$-ornaments and let the corresponding numbers of intersections with the diagonal $y=(k-1)x$ be given by a fixed sequence $(q_1,\ldots,q_a)$ with $q_1+\ldots+q_a=p$. From how many distinct $k$-ornaments intersecting the diagonal $y=(k-1)x$ precisely $p$ times can this $k$-ornament field be obtained by the cutting procedure described above? Naturally, this is equivalent to asking how many different cycles on $[p]$ can be cut to obtain a set of $a\geq 2$ cycles with lengths $(q_1,\ldots,q_a)$ and the answer is just given by the number $(a-1)!\prod_{i=1}^a q_i$.

Thus the number $\hat N^{(k,n)}_{p,a}$ is given by $$\hat N^{(k,n)}_{p,a}=n!\frac{c^{(p)}_{k,n}}{p}\sum_{\substack{(q_1,\ldots,q_a)\in [p]^a\\q_1+\ldots+q_a=p}}\frac{p}{a!\prod^a_{i=1} q_i}$$ and, plugging that in the above expression, we obtain
 $$
[x^n]\log^a G_k(x)=\frac{a!}{n!}\sum_{p=a}^n \left(n!\frac{c^{(p)}_{k,n}}{p}\sum_{\substack{(q_1,\ldots,q_a)\in [p]^a\\q_1+\ldots+q_a=p}}\frac{p}{a!\prod^a_{i=1} q_i}\right)=\sum_{p=a}^n c^{(p)}_{k,n}N_{p,a}.
$$ 
\end{proof}

\section{Combinatorial interpretation via  tree graphs}

\subsection{Tree graphs and associated generating functions}

In this section, we provide an alternative combinatorial interpretation for the logarithm of the binomial series $G_k$ in terms of tree graph structures. To do so, we introduce several sets of labeled graphs.

\begin{definition}[Rooted plane trees]\label{def:1}Let $k\geq 2$. For a finite set $V\subset \N$, we define a rooted plane $k$-ary tree with the vertex set $V$
as follows:  Consider a quadruple $(V,E,r,(\ell(v))_{v\in V})$ such that
\begin{enumerate}
\item $r\in V$, $E\subset \binom{V}{2}$,
\item the graph $(V,E,r)$ is a tree rooted in $r$,
\item for each vertex $v\in V$, the set $C(v)\subset V$ of children of $v$ in $(V,E,r)$ satisfies the constraint $\vert C(v)\vert\leq k$,
\item for each vertex $v\in V$, $\ell(v):C(v) \to \{1,\ldots, k\}$ is an injective map. 

\end{enumerate}

For each vertex $v\in V$, we interpret the numbers $\{1,\ldots, k\}$ as an ordered list of slots potentially available for the children of $v$. We say that the $p^\text{th}$ $v$-slot is occupied by a vertex $j\in V$, if $j\in C(v)$ and $\ell(v)(j)=p\in\{1,\ldots, k\}$. We say that the $p^\text{th}$ $v$-slot is vacant, if such a $j$ does not exist. The slots $\{1,\ldots, k\}$ are visualized in an increasing order from left to right and vacant slots are depicted by small solid (unlabeled) nodes.

We denote the set of rooted plane $k$-ary trees with the vertex set $V$ by $\mathscr T_k(V)$. 
\end{definition}

\begin{rem}
Vacant slots can be interpreted as unlabeled leaf vertices (compare to the \emph{full binary trees} as in~\cite{hiltped91}).
\end{rem}
\begin{figure}[ht]
\centering
\includegraphics[scale=0.4]{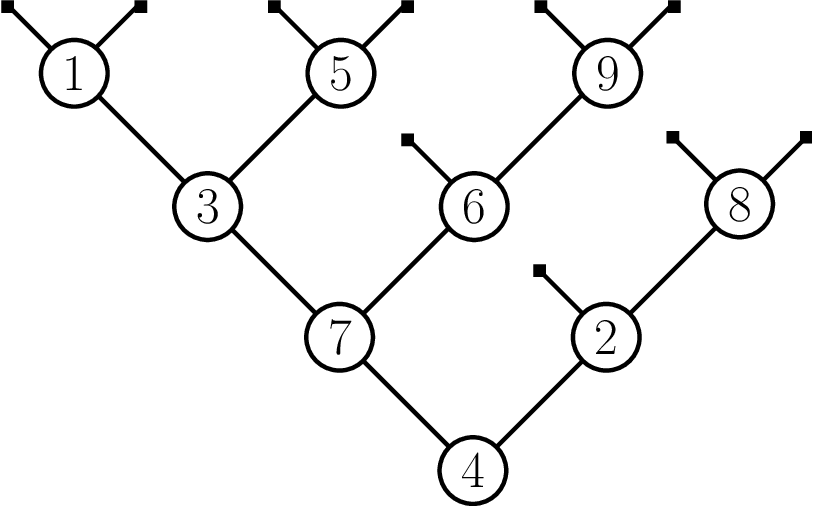}
\caption{Binary ($k=2$) tree with $n=9$ vertices.}\label{binary_tree(n9)}
\end{figure}

\begin{definition}[Root-minimal plane trees]
Let $k\geq 2$. For a finite set $V\subset \N$, let $t$ be a rooted plane $k$-ary tree with the vertex set $V$, i.e., $t\in\mathscr T_k(V)$. We say that a vertex $v\in V$ is \emph{on the rightmost branch of $t$} if $v$ is an element of the vertex set $B\subset V$ defined via the following induction:
\begin{enumerate}
\item Let the root $r\in V$ be in $B$.
\item If a vertex $v\in V$ is in $B$, then the vertex occupying the $k^{\text{th}}$ (rightmost) $v$-slot is in $B$.
\end{enumerate}

We call $t$ \emph{root-minimal} if the root $r\in V$ is smaller (with respect to the canonical order on the natural numbers) than any of the other vertices on the rightmost branch of the tree. The set of root-minimal plane $k$-ary trees is denoted by $\mathscr T_k^{\min}(V)$. We denote the set of root-minimal plane $k$-ary forests with the vertex set $V$ by $\mathscr F_k^{\min}(V)$.
\end{definition}

\begin{figure}[ht]
\centering
\includegraphics[scale=0.4]{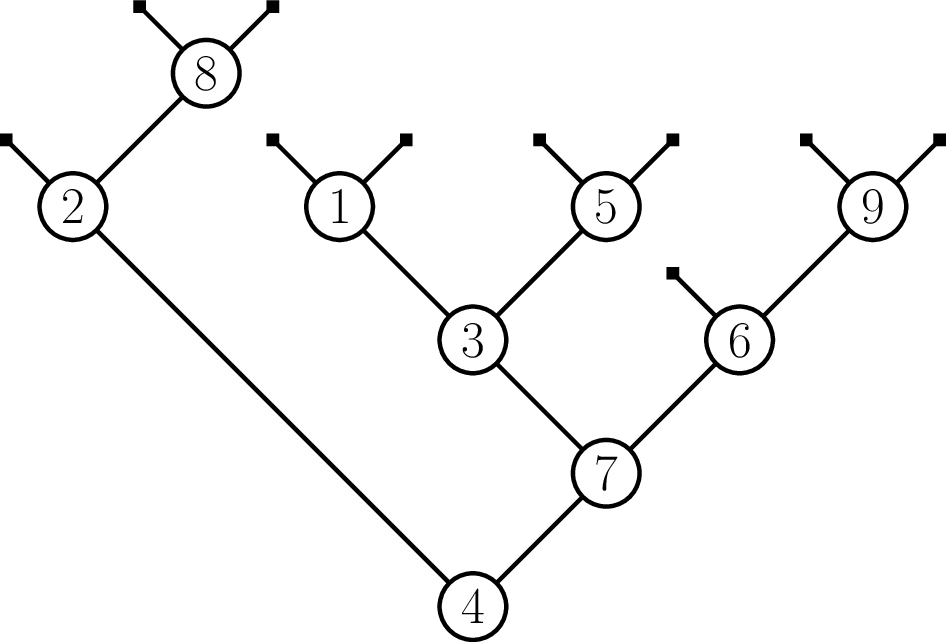}
\caption{Root-minimal binary tree with $n=9$ vertices.}\label{min_rooted_tree(n9)}
\end{figure}

\begin{definition}[Cycle-rooted plane trees] Let $k\geq 2$. For a finite set $V\subset \N$, we define a cycle-rooted plane $k$-ary tree with the vertex set $V$ as follows: Consider a quintuple $(V,E,R, o,(\ell(v))_{v\in V})$ such that

\begin{enumerate}
\item $R\subset V$, $E\subset\binom{V}{2}$,
\item $(R, E\cap\binom{R}{2})$ is the cycle graph associated with the cyclic permutation $o$ on $R$ and is visualized as oriented clockwise,
\item the graph $(V,E\backslash\binom{R}{2},R)$ is a forest of $\vert R\vert$ trees rooted in vertices from $R$,
\item for each vertex $v\in V$, the set $C(v)\subset V$ of children of $v$ in $(V,E\backslash\binom{R}{2},R)$ satisfies the constraint $\vert C(v)\vert\leq k$,
\item for each vertex $v\in V$, $\ell(v):C(v) \to \{1,\ldots, k\}$ is an injective map; we  use the same vocabulary and interpret $\ell(v)$ in the same manner as in Definition~\ref{def:1},

\item For every $r\in R$, the $k^\text{th}$ (rightmost) $r$-slot is vacant.

\end{enumerate}
We denote the set of cycle-rooted plane $k$-ary trees with the vertex set $V$ by $\mathscr T_k^\circ(V)$ and the set of cycle-rooted $k$-ary forests with the vertex set $V$ by $\mathscr F_k^\circ(V)$. 
\end{definition}
\begin{rem}\label{rem:cycle_equiv}
Cycle-rooted trees can be interpreted as equivalence classes of rooted plane trees: Two rooted plane trees are equivalent if and only if they result in the same cycle-rooted tree when we identify the root of the tree with its right-most leaf (the right-most branch therefore becoming the cycle sub-graph in the resulting cycle-rooted tree). In this way a cycle-rooted tree with a cycle of length $r$ corresponds to an equivalence class consisting of $r$ rooted plane trees. Compare this to the definition of $k$-ornaments (Definition~\ref{def:ornaments} in Section~\ref{sec:2}).
\end{rem}
\begin{figure}[ht]
\centering
\includegraphics[scale=0.4]{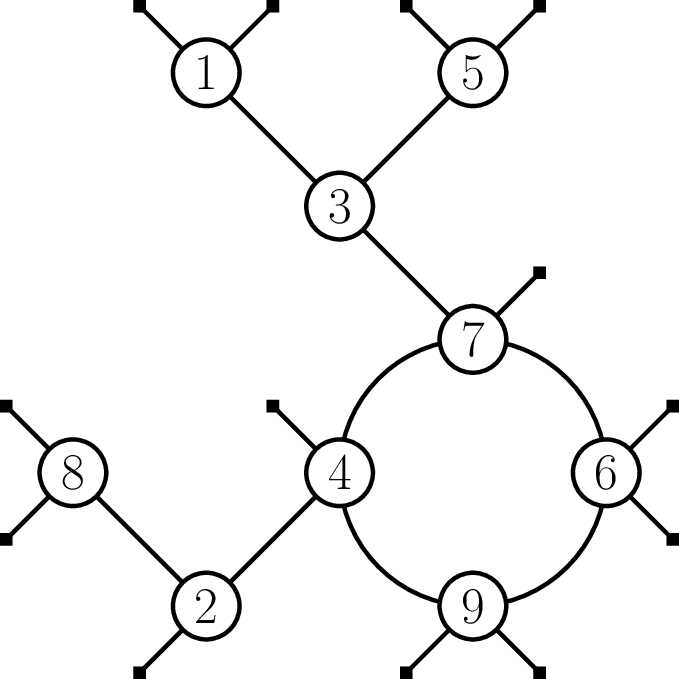}
\caption{Cycle-rooted binary tree with $n=9$ internal vertices.}\label{fig:cycle_rooted_tree(n9)}
\end{figure}
Let $n\in\N$ and let $k\geq 2$. For notational convenience, we set $\mathscr T_k(n):=\mathscr T_k([n])$ and, analogously, write $\mathscr T^{\min}_k (n)$, $\mathscr T^\circ_k (n)$, $\mathscr F^{\min}_k (n)$ and $\mathscr F^\circ_k (n)$ for any $n\in\N$.

As mentioned in the introduction, the exponential generating function for $(\mathscr T_k(n))_{n\in\N}$ is given by the series $G_k$ (see~\cite{hiltped91}), i.e.,
$$
G_k(x)=1+\sum_{n\geq 1}\frac{x^n}{n!}\vert\mathscr T_k(n)\vert.
$$  
Moreover, we denote 
\begin{itemize}

%

\item by $\hat G^{\min}_k$ the exponential generating function for $(\mathscr T^{\min}_k(n))_{n\in\N}$ given by 

$$
\hat G^{\min}_k(x)=\sum_{n\geq 1}\frac{x^n}{n!}\vert\mathscr T^{\min}_k(n)\vert,
$$

\item by $\hat G^{\circ}_k$ the exponential generating function for $(\mathscr T^{\circ}_k(n))_{n\in\N}$ given by

$$
\hat G^{\circ}_k(x)=\sum_{n\geq 1}\frac{x^n}{n!}\vert\mathscr T^{\circ}_k(n)\vert. 
$$
\end{itemize}

\subsection{Bijective results.}

The following lemma is the tree analogue of Lemma~\ref{lem:bb}.
\begin{lemma}\label{lem:bij:tree_to_forest}
Let $n\in\mathbb{N}$ and $k\geq 2$. There is a bijection between the set of $k$-ary trees with $n$ vertices $\mathscr T_k(n)$ and the set of root-minimal $k$-ary forests with $n$ vertices $\mathscr F^{\min}_d(n)$.
\end{lemma}
\begin{proof}
We consider the following map $m$ from $\mathscr T_k(n)$ to $\mathscr F^{\min}_k(n)$. Let $t\in \mathscr T_k(n)$, then we obtain the forest $m(t)\in \mathscr F^{\min}_k(n)$ from $t$ by the following inductive procedure: 
\begin{enumerate}
\item [Step $0$:] \quad Set $i=r$. Set $l=i$.
\item [Step $N\geq 1$:] \quad Let $\# N$ be the number of trees after step $N-1$.

\begin{itemize}
\item If the $k^\text{th}$ (rightmost) $l$-slot is vacant, STOP. 

\item If the $k^\text{th}$ (rightmost) $l$-slot is occupied by a vertex $j\in V$ and $j<i$, then delete the edge $\{l,j\}$, obtaining $\#N+1$ trees, and leave the $k^\text{th}$ $l$-slot vacant. Let all vertices that were roots in the previous step remain roots and let $j$ become the root in the tree to which it belongs. Set $i=j$, $l=i$  and GOTO Step $N+1$.

\item If the $k^\text{th}$ (rightmost) $l$-slot is occupied by a vertex $j\in V$ and $j>i$, then do nothing and the number of trees remains $\#N$. All vertices that were roots in the previous step remain roots. If the $k^\text{th}$ $j$-slot is vacant, STOP. If the $k^\text{th}$ $j$-slot is occupied by some vertex, set $\l=j$ and GOTO Step $N+1$. 
\end{itemize}
\end{enumerate}

Naturally, this procedure produces a forest of root-minimal trees while preserving the vertex set and the offspring constraint $k$, thus the map $m:\mathscr T_k(n)\to\mathscr F^{\min}_k(n)$ is well-defined.

Conversely, given a $k$-ary forest in $\mathscr F^{\min}_k(n)$, one can obtain a tree from it by the following procedure: Order the trees of the forest decreasing in the root numbers (with respect to the canonical order on the natural numbers). From this sequence of trees, we obtain a single tree (with the root given by the largest of the initial roots) by successively attaching the successor tree to the predecessor tree as follows: Let $j$ be the last vertex on the rightmost branch of the predecessor tree. We place the root of the successor tree in the vacant $k^\text{th}$ (rightmost) $j$-slot, leaving the offspring structure unchanged otherwise.  This procedure preserves the vertex set and the offspring constraint $k$ as well, and thus defines a map from $\mathscr F^{\min}_k(n)$ to $\mathscr T_k(n)$ --- which clearly is the inverse for the map $m$ defined above. 
\end{proof}

\begin{rem}Naturally, our choice of root-minimal trees is somewhat arbitrary in the following sense: In the proof of Lemma~\ref{lem:bij:tree_to_forest}, one can choose a different rule to compare the labels $i$ and $j$. E.g., one could instead consider ``maximal-rooted'' trees (or, more generally, use any arbitrary order on the natural numbers instead of the canonical one).
\end{rem}

\begin{figure}[ht]
\centering
\includegraphics[scale=0.4]{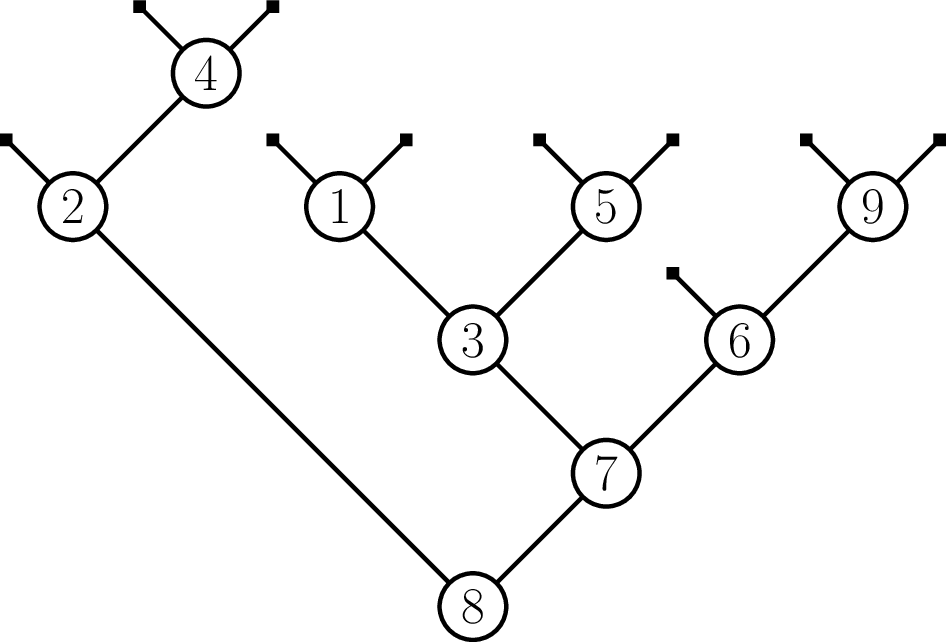}\qquad\qquad
\includegraphics[scale=0.4]{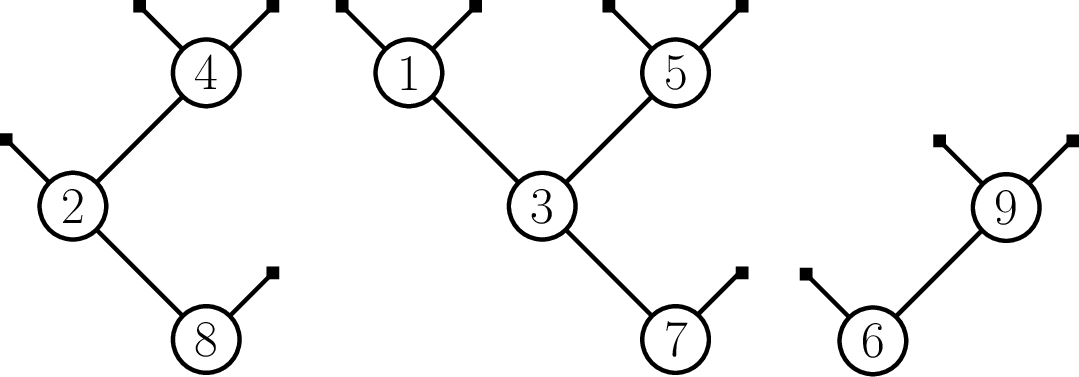}
\caption{On the left side, we see a binary tree with $n=9$ vertices; on the right side, we see the root-minimal binary forest corresponding to it in the sense of the proof of Lemma~\ref{lem:bij:tree_to_forest}.}
\label{fig:lemma_3.1}
\end{figure}

The following theorem is the tree analogue of Theorem~\ref{thm:bij_paths} and a direct consequence of the preceding lemma.
\begin{theorem}\label{thm:main1}Let $k\geq 2$. The following holds as an identity between formal power series:
$$
\log G_k=\hat G^{\min}_k.
$$
\end{theorem}
\begin{proof}Analogously to the proof of Theorem~\ref{thm:bij_paths}, the claim follows directly from Lemma~\ref{lem:bij:tree_to_forest} via a standard combinatorial argument (see, e.g.,~\cite{faris2010combinatorics}, for the argument formulated in the framework of combinatorial species). 
\end{proof}

The following lemma is the tree analogue of Lemma~\ref{lem:bij:orn_to_min}.
\begin{lemma}\label{lem:bij:circ_to_min}
Let $n\in\N$ and $k\geq 2$. There is a bijection between the sets $\mathscr T_k^\circ(n)$ and $\mathscr T^{\min}_k(n)$.
\end{lemma}
\begin{proof}
We consider the following map $p$ from $\mathscr T_k^\circ(n)$ to $\mathscr T^{\min}_k(n)$. Starting with a cycle-rooted tree $c\in\mathscr T_k^\circ(n)$, one obtains a root-minimal tree $p(c)\in\mathscr T^{\min}_k(n)$ by the following procedure: For every vertex $r\in R$ on the unique cycle in $c$, the $k^\text{th}$ (rightmost) $r$-slot is vacant by definition. Delete the edge $\{i,j\}$ of the cycle which connects the minimal cycle vertex $i\in R$ with its neighbor in the counter-clockwise direction $j\in R$. Let the minimal cycle vertex $i$ now be the root of the resulting tree and, for every $r\in R\backslash\{j\}$, let the $k^\text{th}$ $r$-slot be occupied by the former clockwise neighbor of $r$ on the cycle while leaving the $k^\text{th}$ $j$-slot vacant. That way, the former cycle becomes the rightmost branch of the resulting tree. Otherwise, let the offspring structure be inherited from $c$.  Notice that the resulting tree is indeed in $\mathscr T^{\min}_k(n)$, the map $p$ is thus well-defined.

Conversely, to obtain from a root-minimal tree $t \in\mathscr T^{\min}_k(n)$ a cycle-rooted tree in $\mathscr T_k^\circ(n)$ consider the following procedure: Add an edge between the root $r$ of $t$ and the last vertex of the rightmost branch of $t$, obtaining a cycle. Set $R\subset V$ to be the cycle nodes (that are precisely the vertices on the right-most branch of the original root-minimal tree $t$). For every cycle node $r\in R$, let the $k^\text{th}$ (rightmost) $r$-slot be vacant. Otherwise, for every $v\in V$, let the offspring structure of $v$ be inherited from the map $\ell(v)$ defining $t$. Clearly, this procedure provides the inverse to the map $p$ defined above. 
\end{proof}
\begin{rem}
If we identify the cycle-rooted trees with equivalence classes of trees as hinted in Remark~\ref{rem:cycle_equiv}, then a bijection is given by just assigning to a root-minimal tree $t$ its equivalence class $[t]$. The map is indeed invertible, since every equivalence class has a unique representative which is root-minimal (compare to the proof of Lemma~\ref{lem:bij:orn_to_min}).
\end{rem}

\begin{figure}[ht]
\centering
\includegraphics[scale=0.33]{cycle_rooted_tree_n9__equi.eps}\qquad
\includegraphics[scale=0.33]{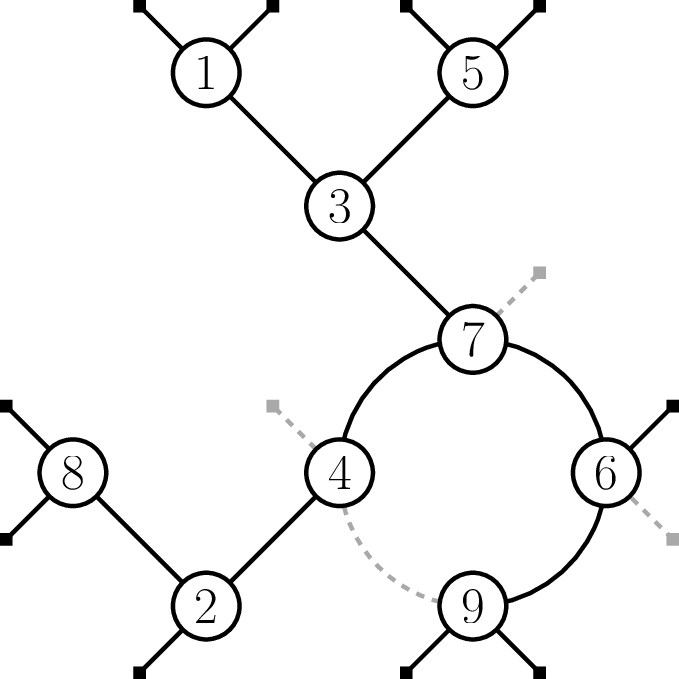}\qquad
\includegraphics[scale=0.33]{min_rooted_tree_n9__equi.eps}
\caption{The cycle-rooted tree from Figure~\ref{fig:cycle_rooted_tree(n9)} (depicted on the left side) corresponds to the root-minimal tree from Figure~\ref{min_rooted_tree(n9)} (depicted on the right side) in the sense of the proof of Lemma~\ref{lem:bij:circ_to_min}. The construction is illustrated in the middle.}
\label{fig:lemma_3.3}
\end{figure}

The following theorem follows immediately from Lemma~\ref{lem:bij:circ_to_min} and Theorem~\ref{thm:main1}. It is the tree analogue of Theorem~\ref{thm:bij_paths_2}:

\begin{theorem}\label{thm:id:logC_circ_trees}Let $k\geq 2$. The following holds as an identity between formal power series:
$$\log G_k=\hat G^{\circ}_k.$$
\end{theorem}
\begin{proof}
The claim follows from Theorem~\ref{thm:main1} and Lemma~\ref{lem:bij:circ_to_min} since the latter implies that $\hat G^{\circ}_k=\hat G_k^{\min}$ for $k\geq 2$.
\end{proof}

We have shown how taking the logarithm of the generating function for $k^{\text{th}}$ Catalan numbers $G_k$ can be interpreted on the level of trees. By Theorem~\ref{thm:main1}, $\log G_k$ can be interpreted as the exponential generating function for root-minimal plane $k$-ary trees --- i.e., taking the logarithm of $G_k$ corresponds to discarding those $k$-ary trees that have roots that are not minimal among the vertices on the right-most branch of the tree. Alternatively, by Theorem~\ref{thm:id:logC_circ_trees}, $\log G_k$ can be interpreted as the exponential generating function for cycle-rooted $k$-ary trees --- so that taking the logarithm corresponds to identifying those trees that result in the same cycle-rooted tree when their right-most branch is ``bent into a circle''.

\appendix
\section{Cyclic multisets: Encoding lattice ornaments and trees}

Here we introduce a way to encode both $k$-ornaments and  cycle-rooted $k$-ary trees by structures we call cyclically ordered multisets. The rough idea of the encoding is best explained starting from binary rooted trees. Each internal vertex (except for the root) sits on a branch connecting one of its leaf-descendants to the root, and is at the origin of a new branch emanating from it. Enumerating the vertices in the order in which they are visited by a depth-first search, along with the lengths of the associated emanating branches, we obtain sequences $(v(1),\ldots,v(n))$, $(f(1),\ldots,f(n))$ of labels and branch lengths, with the branch lengths summing up to the total number of vertices. In turn, the branch lengths may be reinterpreted as step heights of lattice paths.  Alternatively, we may view the branch lengths $f(j)$ as multiplicities of the element $v(j)$ in some multiset. The precise constructions are more involved as $k$-ary trees may have more than one branch emanating from internal vertices and the natural structure for cycle-rooted trees is a cycle, rather than an ordered list, of the vertex labels. 

For every $n\in\N$ and $k\geq 2$, we will introduce a bijective map $\pi$ encoding $[n]$-labeled $k$-ornaments and a bijective map $\tau$ encoding cycle-rooted $k$-ary trees on $[n]$ using the same set of cyclically ordered multisets. Naturally, those maps $\tau$ and $\pi$ induce a bijection between the sets $\mathscr T^\circ_k(n)$ and $\mathscr P^\circ_k(n)$ for every $n\in\N$ and $k\geq 2$ which can be interpreted as a way to encode $k$-ary trees by monotone lattice paths and is similar the well-known encoding of binary trees by Dyck paths from~\cite{pitman2006combinatorial}. Moreover, the bijections $\pi$ and $\tau$ provide an alternative approach to finding the coefficients of $\log G_k$ --- by simply counting cyclically ordered multisets in the image of $\tau$ and $\pi$. Before we further discuss the encoding, we would like to introduce the set of cyclically ordered multisets rigorously:

\begin{definition}[Cyclically ordered multisets]\label{def:co_mult}
Let $k\geq 2$. A cyclically ordered $k$-multiset $(\sigma,f)$ on $[n]$ consists of a cycle (cyclic permutation) $\sigma$ on $[n]$ together with a map $f:[n]\to\N_0^{k-1}$ given by $$[n]\ni i\mapsto (f_1(i),\ldots,f_{k-1}(i))\in\N_0^{k-1}$$ such that $\sum_{i=1}^n\sum_{q=1}^{k-1} f_q(i)=n$. 
To the cycle $\sigma$, assign the cycle graph $C_\sigma=(V,E)$, given by $$V=[n]\times[k-1]$$ and $$E=\{\{(i,q),(j,p)\}\vert\ i=j\text{ and } \vert q-p \vert =1 \text{ or }i\text{ is
}\sigma\text{-predecessor of }j,\ q=k-1 \text{ and  }p=1\}.$$ Alternatively, one can view $f$ as a function on the nodes of $C_\sigma$, i.e., $f: [n]\times[k-1]\to \mathbb N_0$, $(i,q)\mapsto f_q(i)$. We denote the set of cyclically ordered $k$-multisets on $[n]$ by $\mathscr M^\circ_k(n)$.
\end{definition}
%
%

Let $n\in \N$. In the binary case $k=2$, one needs the whole set $\mathscr M^\circ_2(n)$ to encode the corresponding 2-ornaments or binary trees. For $k\geq 3$, however, the set $\mathscr M^\circ_k(n)$ is too big. We introduce a subset of $\mathscr M^\circ_k(n)$ which is naturally suited to encode the structures from $\mathscr P^\circ_k(n)$ and $\mathscr T^\circ_k(n)$:  

\begin{definition}[Multisets with root vertices] Let $m=(\sigma,f)\in \mathscr M^\circ_k(n)$, let $i,j\in[n]$ and let $1\leq k_i,k_j\leq k-1$. We call a simple 
 path on the circle graph $C_\sigma$ starting in $(i,k_i)$ and ending in $(j,k_j)$ a \emph{segment of $C_\sigma$} if $i=j$ and $k_i\leq k_j$ or if it is consistent with the orientation of $\sigma$, i.e., if $i\neq j$ and for every pair of consecutive points $(\ell_1,k-1), (\ell_2,1)$ in $s$ we have that $\ell_2$ is the $\sigma$-sucessor of $\ell_1$. 
To any segment $s$ of $
C_\sigma$ we assign the scope of $s$ given by $$\lambda(s)=\vert\{i\in[n]\vert (i,q)\in s \text{ for some }1\leq q\leq k-1\}\vert$$ and the weight of $s$ in $m$ given by $$w^{(m)}(s)=\sum_{(i,q)\in s} f_{q}(i).$$
For $m\in \mathscr M^\circ_k(n)$, we define the set of root vertices $W(m)$ by 
$$
W(m):=\{i\in[n]\vert \text { every segment } s\text{ of }C_\sigma\text{ starting in }(i,1) \text{ satisfies } w^{(m)}(s)\geq \lambda(s)\}.
$$
We denote the set of those multisets in $\mathscr M^\circ_k(n)$ that possess root vertices by $M(k,n)$, i.e.,
$$
M(k,n):=\{m\in\mathscr M^\circ_k(n)\vert W(m)\neq \emptyset\}.
$$ 
\end{definition}

\begin{figure}[ht]
\centering
\includegraphics[scale=0.3]{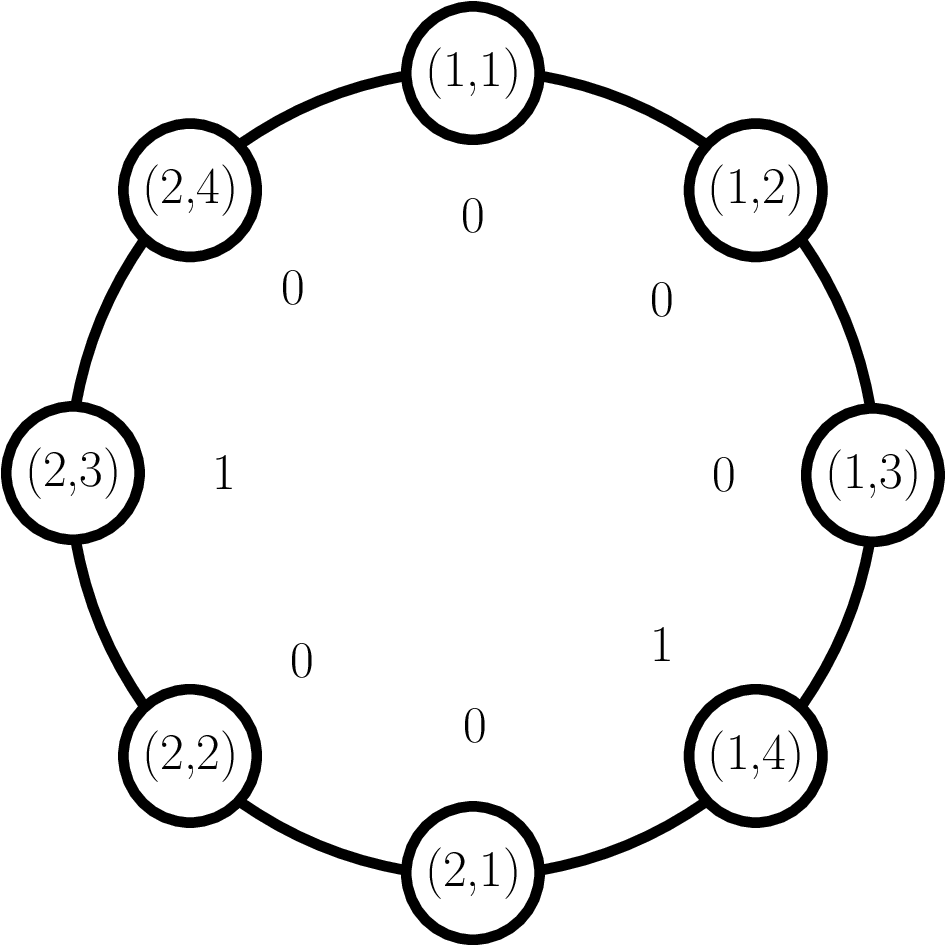}\qquad\qquad
\includegraphics[scale=0.3]{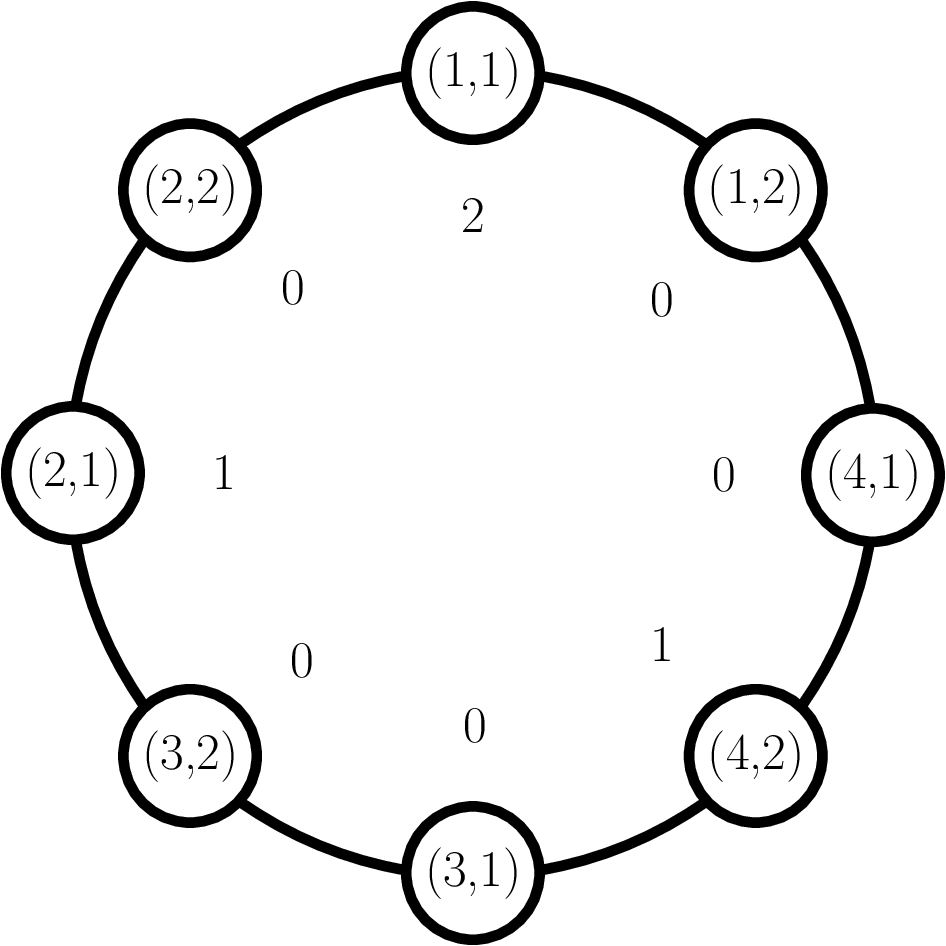}
\caption{On the left side $m\in\mathscr M^\circ_5(2)$ is depicted, on the right side $m'\in \mathscr M^\circ_3(4)$. The numbers inside the circle graph depict the multiplicities of the vertices of the circle graph $C_\sigma$ closest to them. Notice that $m\notin M(5,2)$, but $m'\in M(3,4)$, since $1,2\in W(m')$.}\label{fig:ex:multisets}
\end{figure}
Now we can introduce a map encoding lattice ornaments by cyclically ordered multisets:
\begin{definition}[Map $\pi$ encoding lattice ornaments by multisets]\label{def:pi}
Let $k\geq 2$ and $n\in \N$. We define the embedding $\pi:\mathscr P^\circ_k(n)\to \mathscr M^\circ_k(n)$ as follows: For a $k$-ornament $O\in\mathscr P_k^\circ(n)$, we set $\pi(O)=(\sigma, f)$, where $\sigma$ is simply given by the labeling of $O$. To obtain the map $f$, take any representative of $O$ and set $f_q(i)$, $q\in[k-1]$, $i\in[n]$, to be the number of steps to the right at the height $y=y_i+q-1$, where $y_i$ is the height labeled by $i$ in $O$. 
\end{definition}

\begin{rem}\label{rem:vir_1}
Naturally, the map $\pi$ is indeed injective. The property of the path $O$ to not rise above the diagonal $y=(k-1)x$ corresponds to the property $W(\pi(O))\neq \varnothing$ on the level of multisets. Moreover, the set of labels marking the heights  at which $O$ intersects the diagonal becomes the set $W(\pi(O))$. Thus the range $\pi(\mathscr P^\circ_k(n))$ of $\pi$ is given by $M(k,n):=\{m\in \mathscr M^\circ_k(n)\vert \ W(m)\neq \varnothing\}$ so that $\vert \mathscr P^\circ_k(n) \vert=\vert M(k,n)\vert$. For $k=2$, we have $M(k,n)=\mathscr M^\circ_k(n)$ and $\pi$ is a bijection.
\end{rem}
\begin{figure}[ht]
\centering
\includegraphics[scale=0.5]{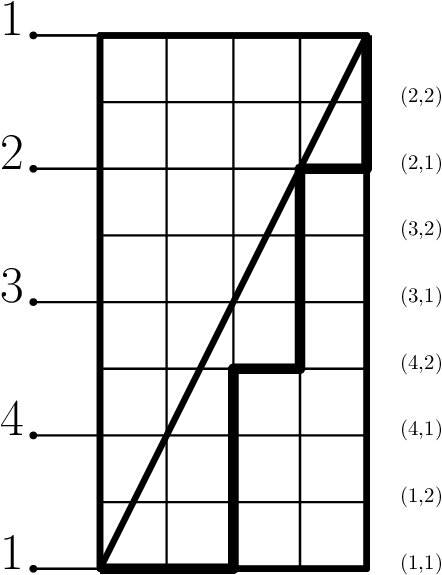}\qquad\qquad
\includegraphics[scale=0.3]{111222.eps}
\caption{The $[4]$-labeled 3-ornament corresponding to the 3-good path depicted on the left side is mapped by $\pi$ to the multiset in $\mathscr M^\circ_3(4)$ depicted on the right side.} \label{fig:ex2:multi}
\end{figure}
Now we investigate how cycle-rooted trees can be encoded by cyclically ordered multiset. To this end, we introduce the following map:

\begin{definition}[Map $\tau$ encoding cycle-rooted trees by multisets]\label{def:tau_step1+2}Let $k\geq2$ and $n\in\N$. We introduce an embedding $\tau:\mathscr T^\circ_k(n)\to \mathscr M^\circ_k(n)$. Given a cycle-rooted tree $t\in\mathscr T^\circ_k(n)$, we construct the cyclically ordered multiset $\tau(t)=(\sigma,f)\in \mathscr M^\circ_k(n)$ by the following two-step procedure:
\begin{itemize}
\item Step 1 (Constructing the cycle $\sigma$ by exploration of vertices in $t$): Starting at any root of $t\in \mathscr T^\circ_k(n)$, the cycle $\sigma$ is obtained by the following exploration procedure: In every step of the exploration, we uncover a single vertex of $t$. In the first step, we uncover an arbitrary root $r$ of $t$. In every further step, as long as there are unexplored vertices in the maximal $k$-ary subtree of $t$ rooted in $r$, we go to the last explored vertex that has an unexplored child and uncover its leftmost unexplored child. When the maximal $k$-ary subtree of $t$ rooted in $r$ is explored, we move to the next root in $t$ according to the cyclic order induced by the oriented cycle of roots $t$ and repeat the procedure. We stop when all vertices of $t$ are explored and define $\sigma$ as the cycle induced directly by the linear order in which the vertices of $t$ were uncovered. 
\item Step 2 (Define the function $f$ by re-distributing multiplicities of vertices in $t$):  Initially every vertex of $t$ is assigned a single multiplicity. Then the multiplicities are re-distributed between the vertices of $t$ by ``rolling-down" (viewed drawing the trees growing upwards with equiangular branches, see Figure~\ref{fig:pi:step1}): Let $i\in[n]$ be an arbitrary vertex of $t$. For $q\in[k]$, consider the path $\Theta_q(i)$ given by the unique simple path starting in $i$ and ending in its leaf-descendant such that every vertex $j\neq i$ on the path occupies slot $q$ of its parent. Denote by $\vert\Theta_q(i)\vert$ the the length of the the path $\Theta_q(i)$, i.e., the number of vertices on $\Theta_q(i)$ excluding $i$.

If $i$ is a root, i.e., $i\in R$, set $$f_q(i):=\vert\Theta_q(i)\vert+\delta_{q,1}$$ for $1\leq q\leq k-1$.

If $i$ is not a root, then $i$ is the child of a vertex, say $i$ occupies slot $p$ of its parent. For $1\leq q\leq k-1$, let $q'$ denote the $q^\text{th}$ smallest element of $[k]\backslash\{p\}$ and set   
$$f_q(i):=\vert\Theta_{q'}(i)\vert.$$

Notice that $\sum_{j=1}^n\sum_{q=1}^{k-1} f_q(j)=n$ indeed holds for the function $f$ defined above.

\begin{figure}[ht]
\centering
\includegraphics[scale=0.3]{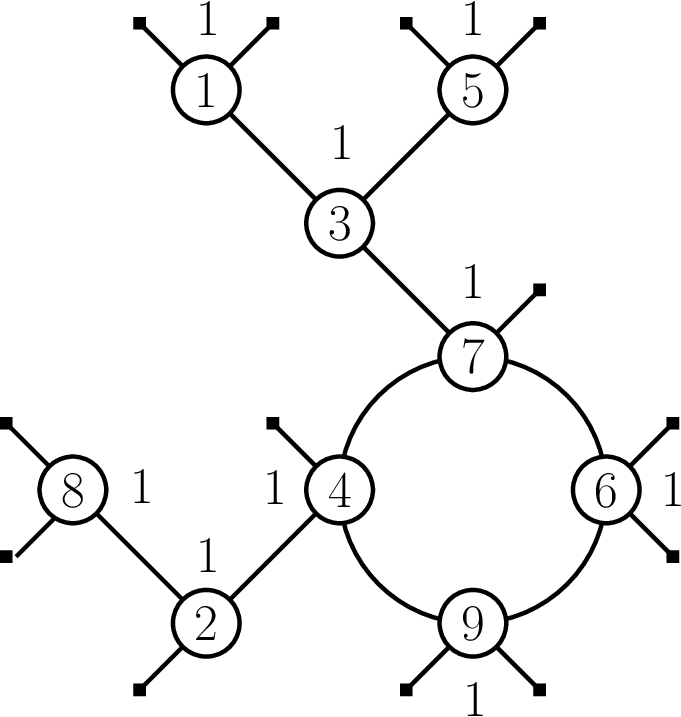}\quad \includegraphics[scale=0.3]{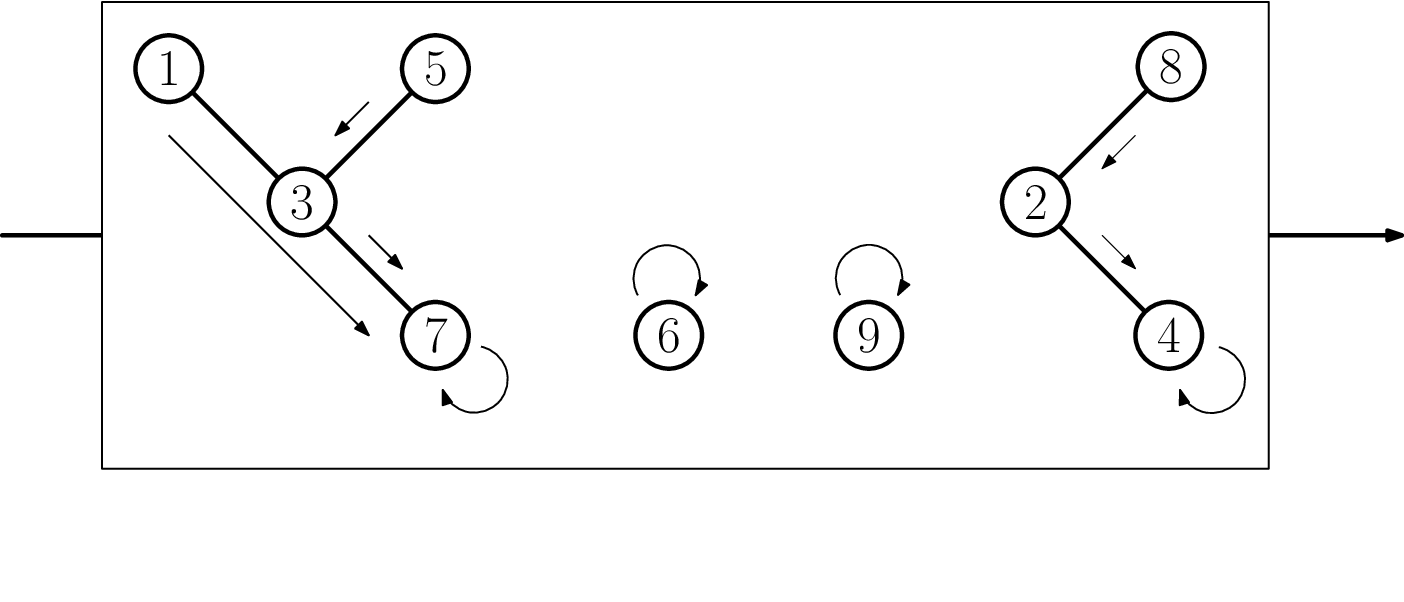}\quad 
\includegraphics[scale=0.3]{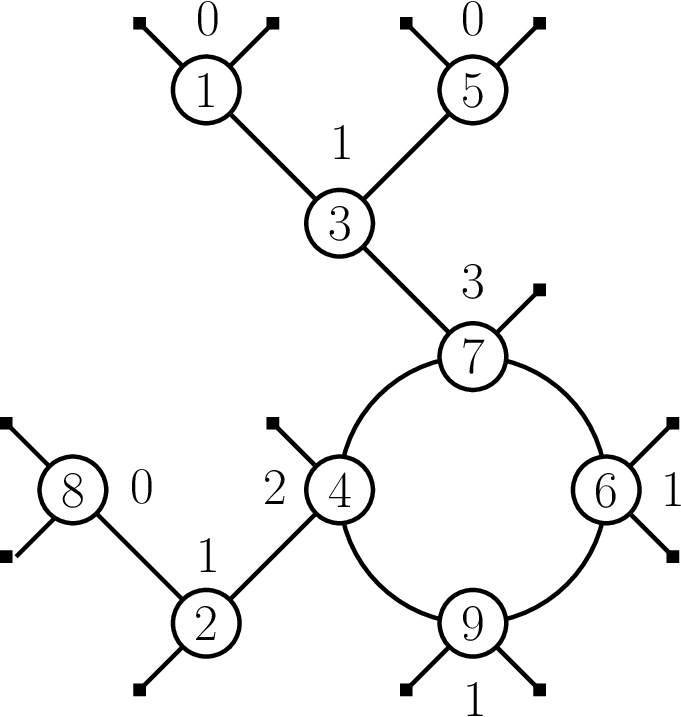}
\caption{Redistribution of multiplicities from Step 2 of Definition~\ref{def:tau_step1+2} in the binary case: The multiplicities of non-root vertices ``roll down" and the multiplicities of roots do not move.}\label{multiset_bin
(n9)}
\label{fig:pi:step1}
\end{figure}

\end{itemize}
\end{definition}
\begin{rem}\label{rem:vir_2}
The map $\tau$ is indeed injective. The set $R$ of roots of $t$ is mapped under $\tau$ precisely onto the set $W(\tau(t))$ on the level of multisets. Again, the range $\tau(\mathscr T^\circ_k(n))$ of $\tau$ is given by $M(k,n)$ so that $\vert\mathscr T^\circ_k(n) \vert =\vert M(k,n)\vert$. In the binary case $k=2$, we have $M(k,n)=\mathscr M^\circ_k(n)$ and $\tau$ is a bijection.
\end{rem}

\begin{figure}[ht]
\centering
\includegraphics[scale=0.4]{cycle_rooted_tree_n9__equi.eps}\qquad
\includegraphics[scale=0.3]{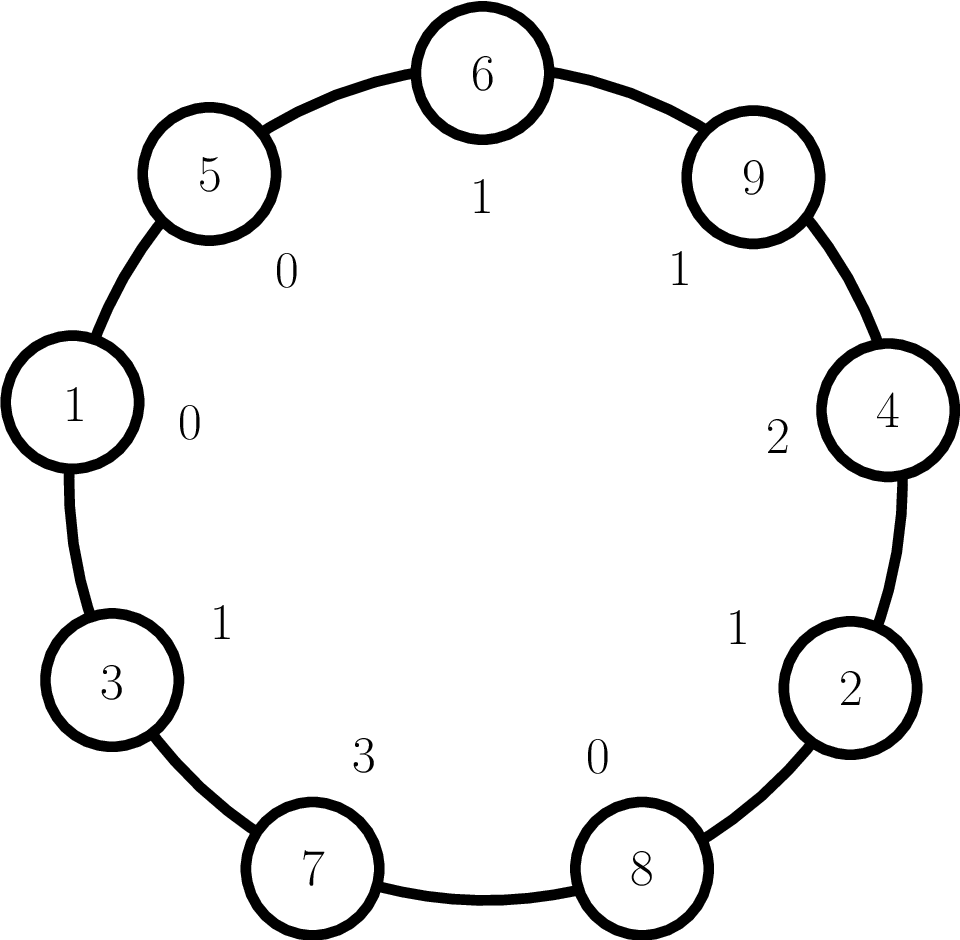}
\caption{Final result: The cycle-rooted tree from Figure~\ref{fig:cycle_rooted_tree(n9)} (depicted on the left side) is mapped by $\tau$ to the multiset from $\mathscr M^\circ_{2}(9)$ (depicted on the right side).}
\label{fig:pi}
\end{figure}

Let $k\geq 2$ and $n\in\N$. By Remark~\ref{rem:vir_1} and Remark~\ref{rem:vir_2}, a bijection between the sets $\mathscr T^\circ_k(n)$ and $\mathscr P^\circ_k(n)$ is given by the composition $\hat\pi^{-1}\circ\tau$, where $\hat\pi: \mathscr P_k^\circ(n)\to M(k,n)$ is given by $\hat\pi(O)=\pi(O)$ for $O\in\mathscr P_k^\circ(n)$. Moreover, let $t\in\mathscr T^\circ_k(n)$ and $O_t:=\hat\pi^{-1}(\tau (t))$, then there is a one-to-one correspondence between the roots of $t$ (vertices $R$ of the cycle subgraph of $t$) and the labels at which $O_t$ intersects the diagonal $y=(k-1)x$. The bijection can be viewed as an alternative to the well-known encoding of binary trees by Dyck paths presented in~\cite[Chapter 6.3]{pitman2006combinatorial} which also involves a depth-first exploration of the tree (as described in Step 2 of Definition~\ref{def:tau_step1+2}).

Finally, notice the following: It can be shown that the set $M(k,n)$ contains exactly the fraction $\frac{1}{k-1}$ of all elements in $\mathscr M_k^\circ(n)$. Since by definition $\vert\mathscr M^\circ_k(n) \vert=(n-1)!\mch{(k-1)n}{n}$ holds, where $\mch{i}{j}$ denotes the multiset coefficient and can be written as $\mch{i}{j}=\binom{i+j-1}{j}$ for $i,j\in\N$, we have 
$$
\vert M(k,n) \vert=\frac{\vert\mathscr M_k^\circ(n)\vert}{k-1}=\frac{(n-1)!}{k-1}\binom{kn-1}{n}=\frac{(kn-1)!}{(kn-n)!}.
$$
 This outlines an alternative proof for Theorem~\ref{thm:orn_count}, since we have $\vert M(k,n) \vert=\vert \mathscr P^\circ_k(n)\vert=\vert\mathscr T^\circ_k(n)\vert$ and thus $\log G_k$ is the exponential generating function for $(M(k,n))_{n\in\N}$, i.e.,
$$
\log G_k(x)=\sum_{n\geq 1}\frac{x^n}{n!}\vert M(k,n) \vert=\sum_{n\geq 1}\frac{x^n}{n!}\frac{(kn-1)!}{(kn-n)!}.
$$ 

\subsection*{Acknowledgements}
S.J. thanks Peter Winkler for discussions during a stay at the
Institute for Computational and Experimental Research in Mathematics
(ICERM) in Providence, RI, during the semester program \emph{Phase
transitions and emerging properties} in spring 2015. L.K. thanks Jan Philipp Neumann for helpful comments on an earlier version of the paper.


\begin{thebibliography}{99}

\bibitem{bergeron-labelle-leroux1998book}
{\sc Bergeron, F., Labelle, G., and Leroux, P.}
\newblock {\em Combinatorial species and tree-like structures}, vol.~67 of {\em
  Encyclopedia of mathematics and its applications}.
\newblock Cambridge University Press, 1998.

\bibitem{chu2019}
{\sc Chu, W.}
\newblock Logarithms of a binomial series: Extension of a series of {K}nuth.
\newblock {\em Math. Commun.}, 24 (2019), 83--90.

\bibitem{drube2022}
{\sc Drube, P.}
\newblock Raised $k$-{D}yck paths.
\newblock {\em J. Integer Seq. 26\/} (2023), Article 23.6.7.

\bibitem{faris2010combinatorics}
{\sc Faris, W.~G.}
\newblock Combinatorics and cluster expansions.
\newblock {\em Prob. Surveys 7\/} (2010), 157--206.

\bibitem{hiltped91}
{\sc Hilton, P., and Pedersen, J.}
\newblock {C}atalan {N}umbers, {T}heir {G}eneralization, and {T}heir {U}ses.
\newblock {\em Math. Intell. 13\/} (03 1991), 64--75.

\bibitem{jansen2015tonks}
{\sc Jansen, S.}
\newblock Cluster and virial expansions for the multi-species {T}onks gas.
\newblock {\em J. Stat. Phys. 161}, 5 (2015), 1299--1323.

\bibitem{knuth_christmas2014}
{\sc Knuth, D.~E.}
\newblock 3/2-ary trees. {A}nnual {C}hristmas lecture, 2014.
\newblock \\ https://www.youtube.com/watch?v=P4AaGQIo0HY.

\bibitem{knuth2015prob}
{\sc Knuth, D.~E.}
\newblock Log-squared of the {C}atalan generating function.
\newblock {\em Amer. Math. Monthly 122:390\/} (2015), Problem 11832.

\bibitem{lin2011}
{\sc Lin, C.-H.}
\newblock Some combinatorial interpretations and applications of
  {F}uss-{C}atalan numbers.
\newblock {\em ISRN Discrete Math.\/} (2011).

\bibitem{pitman2006combinatorial}
{\sc Pitman, J.}
\newblock {\em Combinatorial stochastic processes: {\'E}cole d'\'et{\'e} de
  probabilit{\'e}s de {S}aint-{F}lour {XXXII}-2002}.
\newblock Springer, 2006.

\bibitem{prod2019}
{\sc Prodinger, H.}
\newblock Logarithms of a binomial series: A {S}tirling number approach.
\newblock {\em Ars Math. Contemp. 17\/} (2019), 271--275.

\bibitem{stanley_2015}
{\sc Stanley, R.~P.}
\newblock {\em Catalan Numbers}.
\newblock Cambridge University Press, 2015.

\end{thebibliography}
\end{document}